\documentclass[12pt, a4paper]{amsart}
\usepackage[left=3cm, right=3cm]{geometry}
\usepackage{amsmath,amssymb,amsthm,enumerate}
\usepackage{graphicx}
\usepackage{pstricks-add}
\usepackage{mathrsfs} 
\usepackage{xspace}

\usepackage[pagebackref=true]{hyperref}

\title[A small hyperbolic Kazhdan group]{A sixteen-relator presentation of\\ an infinite hyperbolic Kazhdan group}

\author[P.-E. Caprace]{Pierre-Emmanuel Caprace}
\thanks{P.-E.C. is a F.R.S.-FNRS senior research associate, supported in part by EPSRC grant no EP/K032208/1.}
\address{Universit\'e catholique de Louvain, IRMP, Chemin du Cyclotron 2, bte L7.01.02, 1348 Louvain-la-Neuve, Belgique}
\email{pe.caprace@uclouvain.be}

\date{August 31, 2017}

\newtheorem{thm}{Theorem}

\newtheorem{prop}[thm]{Proposition}

\newtheorem{cor}[thm]{Corollary}

\theoremstyle{definition}

\newtheorem{rmk}[thm]{Remark}
\newtheorem{example}[thm]{Example}

\newcommand{\triv}{\{1\}}

\begin{document}

\begin{abstract} 
We provide an explicit presentation of an infinite hyperbolic Kazhdan group with $4$ generators and $16$ relators of length at most $73$.	That group acts properly and cocompactly on a hyperbolic triangle building of type $(3,4,4)$. We also point out a variation of the construction that yields examples of lattices in $\tilde A_2$-buildings admitting non-Desarguesian residues of arbitrary prime power order. 
\end{abstract}
\maketitle


\section{Hyperbolic Kazhdan groups}

\bigskip

\begin{flushright}
\itshape\small
On ne peut pas faire plus concis!\upshape\\
\smallskip
Raymond Devos, \emph{Mati\`ere \`a rire}, 1991
\end{flushright}

\bigskip

The existence of infinite Gromov hyperbolic groups enjoying Kazhdan's property~(T) has been known since the origin of the theory of hyperbolic groups, as a combination of the following  results.
\begin{itemize}
	\item Every simple Lie group possesses a cocompact lattice, by \cite{Borel}; 
	
	\item the rank one simple Lie groups  $\mathrm{Sp}(n, 1)$ (with $n \geq 2$) and  $F_4^{-20}$ have (T), by \cite{Kostant} (see also \cite[\S3.3]{BHV});
	
	\item if a locally compact group $G$ has property (T), then so does every lattice $\Gamma$ in $G$ by \cite{Kazhdan} (see also  \cite[Theorem~1.7.1]{BHV});
	
	\item a cocompact lattice in a rank one simple Lie group is Gromov hyperbolic, since it is virtually the fundamental group of a closed Riemannian manifold of negative sectional curvature, see  \cite{Gromov}. 	
\end{itemize}
The smallest known dimension of a negatively curved closed manifold $M$ such that $\pi_1(M)$ has (T) is~$8$ (namely when $M$ is covered by the symmetric space of $\mathrm{Sp}(2, 1)$), and I am not aware of any known explicit presentation of the fundamental group $\pi_1(M)$ in that case. This  is a very interesting and natural problem. By the Hyperbolization Theorem (see  \cite[Theorem~1.7.5]{AFW} and references therein),  the fundamental group of a negatively curved closed manifold $M$ of dimension~$1, 2$ or $3$ is a lattice in $\mathbf R$ or  $O(2, 1)$ or $O(3,1)$. Therefore it cannot be a Kazhdan group by \cite[Theorem~2.7.2]{BHV} (see also \cite{Fujiwara} for a more general result on the failure of property (T) for $3$-manifold groups). Whether there exists a negatively curved closed manifold $M$ of dimension $4, 5, 6$ or $7$ with a Kazhdan fundamental group is another intriguing open problem.  M.~Kapovich pointed out to me that the  related problem of finding objects of either of the following kinds, is also open:
\begin{itemize}
	\item a nonpositively curved closed  manifold, not homeomorphic to a locally symmetric space, and   with a Kazhdan fundamental group; 
 \item  a  Kazhdan Poincar\'e duality group not isomorphic to a lattice in a  connected Lie group. 	
\end{itemize}

The possibility to write down an explicit presentation of an infinite hyperbolic Kazhdan group was first realized in \cite[Corollary~2]{BaSw}, where the geometric approach to property (T) via the spectral gap of finite graphs is exploited (see \cite[Chapter~5]{BHV} for an exposition of that approach including a historical account).  The graphs used in \cite{BaSw} are certain Cayley graphs of $\mathrm{SL}_2(\mathbf Z/n\mathbf Z)$, which satisfy the required spectral gap condition for $n$ sufficiently large. An alternative source of finite Cayley graphs that enjoy the required spectral condition is suggested by A.~Valette in his review of \cite{BaSw}, but I am not aware of any reference where that suggestion was incarnated into an explicit presentation of a hyperbolic Kazhdan group.  A different construction is highlighted by M.~
Bourdon in \cite[\S1.5.3]{Bourdon}. It gives rise to cocompact lattices in certain Gromov hyperbolic Fuchsian buildings, and also relies on the geometric approach to property (T). The advantage is that the finite graphs on which the spectral gap condition is tested are finite generalized polygons, and the eigenvalues of their incidence matrix is explicitly known by classical results from \cite{FH}. Nevertheless, the corresponding group presentations one obtains from that construction take several hundreds relations. The variations on Bourdon's construction described in \cite{Swiat} also seem to require a rather large number of relators.  Other examples of infinite hyperbolic Kazhdan groups are studied in \cite{LMW}, but no explicit short presentation is recorded there.

Cornelia Drutu asked me whether it was possible to use buildings in order to construct an explicit short presentation of an infinite hyperbolic group with Kazhdan's property (T). As explained   in \cite[Section~19.8]{DK}:  ``\textit{while `generic' finitely presented groups are infinite and satisfy Property (T), finding explicit and reasonably short presentations presents a bit of a challenge}''. In that context,  targeting hyperbolic buildings   is  especially natural in view of the fact that there exist five-relator presentations of infinite Kazhdan groups acting properly and cocompactly on buildings of type $\tilde A_2$, see \cite[Examples following Theorem~5.8]{Essert}. Note that  those groups cannot be hyperbolic since they are quasi-isometric to a $2$-dimensional Euclidean building.  The shortest presentation I could find in attempting  to answer her question is   the following. 

\begin{thm}\label{thm:main}
	The group 
\begin{align*}
E = \langle x, y, z, t, r \; | \; & x^7, y^7, [x, y]z^{-1}, [x, z], [y, z], \\
& t^2, r^{73}, trtr, \\
&	[x^{2}yz^{-1}, t], [xyz^3, tr], [x^{3}yz^{2}, tr^{17}], \\
& [x, tr^{-34}], [y, tr^{-32}],  [z, tr^{-29}], \\
&[x^{-2}yz, tr^{-25}], [x^{-1}yz^{-3}, tr^{-19}], [x^{-3}yz^{-2}, tr^{-11}] \rangle,
	\end{align*}
	is an infinite Gromov hyperbolic group enjoying Kazhdan's property (T). It acts faithfully, properly, cocompactly (not type-preservingly) on a thick hyperbolic Fuchsian building of type $(3, 4, 4)$. In particular $E$ is quasi-isometrically rigid by \cite{Xie}. 
\end{thm}

In view of the relation $[x, y] =z $, the generator $z$ is redundant, and the  presentation of $E$ given in Theorem~\ref{thm:main} is equivalent to  a presentation with $4$ generators and $16$ relators. This modification increases the length of some of the relators, but one checks that the maximal length of a relator in that $16$-relator presentation of $E$ remains equal to $73$. 

The group $E$ may be viewed as the fundamental group of a simple complex of finite groups, in the sense of \cite[Chapter~II.12]{BH}. The underlying simplicial complex $\mathcal Y$ has $11$ vertices, denoted by $a, b, c_1, \dots, c_9$. It has  $19$ edges and $9$ faces, spanned by $abc_i$ for $i = 1, \dots, 9$. As a metric space, it can be viewed as   nine hyperbolic triangles with angles $\pi/2, \pi/4, \pi/6$, glued along their hypothenuse $[ab]$ (the triangle $abc_i$ is depicted in Figure~\ref{fig:triangle}). The angle at $a$ is $\pi/6$ and the angle at $b$ is $\pi/4$. The vertex group $E_a$ is the dihedral group of order $146$; in the above presentation of $E$ it is generated by $r$ and $t$. The vertex group $E_b$ is the Heisenberg group over $\mathbf F_7$, of order $343$; in the presentation above it is generated by $x$ and $y$.  The vertex group $E_{c_i}$ is the cyclic group of order~$14$  for $i = 1, \dots, 9$. The edge groups $E_{ac_i}$ (resp. $E_{bc_i}$) are cyclic of order $2$ (resp. $7$); they are generated by $t, tr, tr^{17}, tr^{-34}, tr^{-32}, tr^{-29}, tr^{-25}, tr^{-19}$ and $tr^{-11}$ (resp. $x^{2}yz^{-1}, xyz^3, x^{3}yz^{2},x, y, z,  x^{-2}yz, x^{-1}yz^{-3}$ and $x^{-3}yz^{-2}$). The edge group $E_{ab}$ and the face groups $E_{abc_i}$ are trivial.

	\begin{center}
\begin{figure}[h]
	\includegraphics[width=10cm]{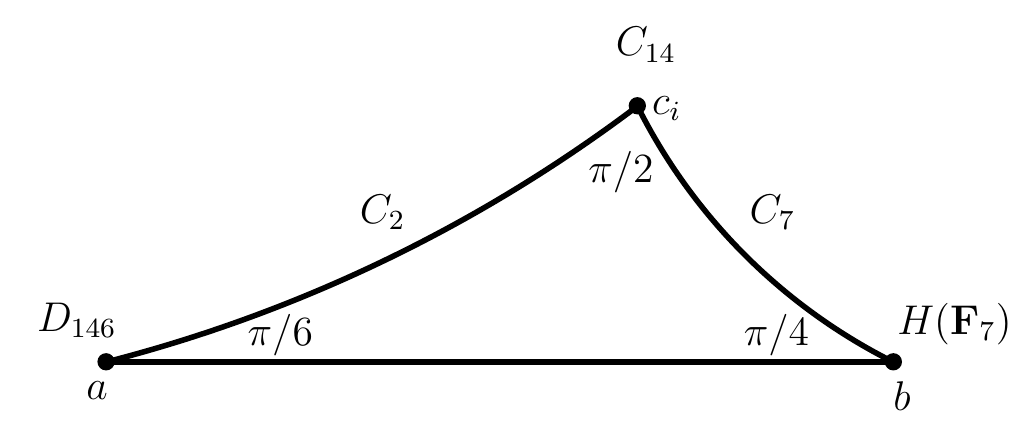}
	\caption{The triangle $abc_i$} \label{fig:triangle}
\end{figure}
	\end{center}

The proof of Theorem~\ref{thm:main} consists of the following steps. 

\begin{itemize}
	\item The link at $a$ in the local development of $\mathcal Y$ around $a$ is the first barycentric subdivision of the incidence graph of the projective plane over the field $\mathbf F_8$ of order $8$.  In particular that link has girth~$12$, so the Link Condition is satisfied at $a$. This step is achieved by Example~\ref{ex:DiffSet} and Proposition~\ref{prop:Dihedral} below. 
	
	\item The link at $b$ in the local development of $\mathcal Y$ around $b$ is the incidence graph of a generalized quadrangle of order~$(8, 6)$. In particular that link has girth~$8$, so the Link Condition is satisfied at $b$. This step is achieved by Corollary~\ref{cor:Heisen} below. 
	
	\item The link at $c_i$ is the complete bipartite graph $K_{2, 7}$, its  girth is~$4$, so the Link Condition is satisfied at $c_i$ for every $i$.
	
	\item By the first three steps, the complex of groups $\mathcal Y$ is developable, and its universal cover $\widetilde{\mathcal Y}$ is a simplicial complex carrying a CAT($-1$) metric on which the fundamental group $E$ acts faithfully, properly and cocompactly by isometric automorphisms, see \cite[Theorem~II.12.28]{BH}. Moreover, since all the links of  $\widetilde{\mathcal Y}$ are $1$-dimensional spherical buildings, it follows from \cite[Theorem~1]{Tits} that $\widetilde{\mathcal Y}$ is a $2$-dimensional building. It is a hyperbolic Fuchsian building of type $(2, 4, 6)$ and the $E$-action on it is type-preserving. However $\widetilde{\mathcal Y}$ is not thick as a building of type $(2, 4, 6)$. If one discards the edges of  $\widetilde{\mathcal Y}$ contained in exactly two $2$-simplices, then  $\widetilde{\mathcal Y}$  becomes a thick building of type $(3, 4, 4)$ on which $E$ acts in a non-type-preserving way. 
	
	\item The final step is to check that $E$ has property (T). This is achieved  using a straightforward computation based on the criterion established by I.~Oppenheim in \cite{Opp}, see  Proposition~\ref{prop:T} below.  This completes the proof of Theorem~\ref{thm:main}.
\end{itemize}

\section{Projective planes and dihedral groups}

We recall that a graph is the incidence graph of a projective plane if and only if it is bipartite, has diameter~$3$ and girth~$6$. We also recall that a \textbf{difference set} in a group $G$ is a subset $\Delta$ of $G$ such that every non-trivial element $g$ of $G$ can be written in a unique way as $g = \sigma^{-1}\tau$ with $\sigma , \tau \in \Delta$. Notice that $G$ must have order $q^2 +q+1$ where $q = |\Delta|-1$. 

\begin{example}\label{ex:DiffSet}
	The set    
	$$\begin{array}{rcl}
	\Delta &= &\{0, 1, 17,39, 41, 44, 48, 54, 62\} \\
	&=& \{0, 1, 17, -34, -32, -29, -25, -19, -11\}
	\end{array}$$ 
	is a difference set in the cyclic group $\mathbf Z/73\mathbf Z$. 	
\end{example}

Every difference set in a group $G$  gives rise to a projective plane; conversely, every projective plane admitting an automorphism group acting sharply transitively on its points gives rise to a difference set. Moreover, for every prime power $q$, the \textbf{Desarguesian plane of order~$q$}, i.e. the projective plane over the field $\mathbf F_q$ of order $q$, has a cyclic automorphism group acting sharply transitively on its points. In particular the cyclic group of order $q^2 +q+1$ has a difference set. We refer to \cite[Lemma~D.1]{BCL} and \cite[pp. 105--106]{Dembowski} for proofs of those assertions. 

Given a group $G$ and a collection $\{P_i \mid i \in I\}$ of subgroups of $G$, the \textbf{coset graph} of $G$ with respect to $\{P_i \mid i \in I\}$ is the bipartite graph whose vertex set is the union of $G$ with $\bigcup_{i \in I} G/P_i$, and where the element $g \in G$ forms an edge with the coset $hP_i$ if and only if $g \in hP_i$. 

\begin{prop}\label{prop:Dihedral}
Let $q$ be a prime power  and let $n = q^2 + q +1$. Let $D_{2n} = \langle r, t \mid r^n, t^2, trtr \rangle$ be the dihedral group of order $2n$, and let $\Delta$ be a difference set in the cyclic group $\mathbf Z/n\mathbf Z$. Then:
\begin{enumerate}[(i)]
\item The Cayley graph of $D_{2n}$ with respect to the set $\{tr^{\sigma} \mid \sigma   \in \Delta\}$ is the incidence graph  of a  projective plane of order $q$. 
\item  The coset graph of $D_{2n}$ with respect to the subgroups $\{\langle tr^{\sigma} \rangle \mid \sigma   \in \Delta\}$ is the first barycentric subdivision of the  incidence graph of a  projective plane of order $q$.  
\end{enumerate}

\end{prop}

\begin{proof}
Since any reflection in $D_{2n}$ has non-trivial image in the quotient $D_{2n}/\langle r \rangle$, it follows that any loop in the Cayley graph $\mathcal G$ of $D_{2n}$ with respect to the set $\{tr^{\sigma} \mid \sigma   \in \Delta\}$ has even length. In particular $\mathcal G$ is bipartite. If $\mathcal G$ contains a loop of length~$4$ through the identity, then there exist $\sigma_1, \dots, \sigma_4 \in \Delta$ with $1 = tr^{\sigma_1} t r^{\sigma_2} t r^{\sigma_3} t r^{\sigma_4}$. Hence $r^{-\sigma_1 + \sigma_2} r^{-\sigma_3 +\sigma_4} = 1$. Since $\Delta$ is a difference set, we must have $\sigma_1 = \sigma_4$ and $\sigma_2 = \sigma_3$, so that the loop was a backtracking path. Thus $\mathcal G$ has girth at least~$6$.  Observing that $\mathcal G$  is a vertex-transitive bipartite graph of degree~$q+1$, we infer   that the  the total number of vertices at distance exactly~$2$ from the identity vertex in  $\mathcal G$ is $q(q+1)$. Since the total number of vertices of $\mathcal G$ is  $2(q^2+q+1)$ and since $\mathcal G$ is bipartite, we deduce that $\mathcal G$ has diameter~$3$ and girth~$6$. This proves assertion~(i). Assertion~(ii) follows from (i) since the coset graph in question is the first barycentric subdivision of $\mathcal G$.
\end{proof}

\section{Generalized quadrangles and Heisenberg groups}

We recall that a graph is the incidence graph of a \textbf{generalized quadrangle} if and only if it is bipartite, has diameter~$4$ and girth~$8$. The \textbf{order} of a finite  generalized quadrangle is the pair $(s, t)$ such that the vertex degrees of the incidence graph of the quadrangle are $s+1$ and $t+1$.  

The following observation is closely related to a  result of W.~Kantor \cite[Theorem~2]{Kantor}. It allows one to recognize when a coset graph (which is automatically bipartite) is the incidence graph of a generalized quadrangle. 

\begin{prop}\label{prop:Kantor}
Let $\mathcal G$ be the coset graph of a group $G$ with respect to a collection  $\{P_i \mid i \in I\}$ of subgroups. 
\begin{enumerate}[(i)]
\item If $P_i \cap P_j = \triv$ for all  distinct $i, j \in I$, then $\mathcal G$ has girth~$\geq 6$. 

\item If $P_iP_j \cap P_k = \triv$ for all distinct $i, j, k \in I$, then $\mathcal G$ has girth~$\geq 8$. 

\item Let $s = | I | -1$ and suppose that $t = |P_i|-1$ for all $i \in I$. If the condition (ii) holds and if in addition $G$ is finite of order $|G| = (1+t)(1+st)$, then $\mathcal G$ is the incidence graph of a generalized quadrangle of order $(s, t)$. 
\end{enumerate}

\end{prop}
\begin{proof}
The proof is a direct computation similar to the proof of Proposition~\ref{prop:Dihedral}.
\end{proof}

The following consequence allows one to recover a family of finite generalized quadrangles that is well-known to the experts; it was first discovered by S.~Payne \cite{Payne}. The right choice of $p+2$ cyclic subgroups was recorded in \cite[Theorem~3.8]{Essert}.

\begin{cor}\label{cor:Heisen}
Let $p$ be an odd prime and  $H(\mathbf F_p) = \langle x, y \mid x^p, y^p, [x, z], [y, z] \rangle$ be the Heisenberg group over $\mathbf F_p$, where $z = [x, y]$. Then the coset graph of $H(\mathbf F_p)$ with respect to the collection  $\{\langle x \rangle, \langle z \rangle\} \cup \{\langle x^a y z^{-\frac a 2}\rangle \mid a=0, \dots p-1\} $  of $p+2$ cyclic subgroups of order $p$ is the incidence graph of a generalized quadrangle of order $(p+1, p-1)$. 
\end{cor}
\begin{proof}
One readily checks that the conditions from Proposition~\ref{prop:Kantor} are satisfied. 
\end{proof}

\section{The spectral criterion for property (T)}

The following criterion for property (T) follows easily   from the main result of \cite{Opp}. 

\begin{prop}\label{prop:T}
Let $X$ be a hyperbolic Fuchsian building of type $(3, 4, 4)$ and $\Gamma$ be a discrete group acting propertly, cocompactly on $X$ by automorphisms. Assume that the projective plane residues of $X$ have order $p+1$ and that the generalized quadrangle residues of $X$ have order $(p+1, p-1)$. If $p \geq 6$,  then $\Gamma$ has Kazdhan's property (T). 
\end{prop}
\begin{proof}
We recall from \cite{FH} that the smallest positive eigenvalue of the Laplacian of the incidence graph of a projection plane of order $p+1$ (resp. a generalized quadrangle of order $(p+1, p-1)$) is $\lambda_P = 1 - \frac{\sqrt{p+1}}{p+2}$ (resp. $\lambda_Q = 1- \sqrt{\frac 2 {p+2}}$). By \cite[Theorem~1]{Opp}, the group $\Gamma$ has property (T) provided that the following two conditions hold:
\begin{itemize}
\item $\lambda_P + 2 \lambda_Q > 3/2$,

\item $(\lambda_P + \lambda_Q -1)^2 + 2(\lambda_P + \lambda_Q-1)(2\lambda_Q -1) >0$.
\end{itemize}
A straightforward computation shows that the first condition holds for all integer $p \geq 5$, while the second holds for all $p \geq 6$.   
\end{proof}

\section{Variations on the same theme}

There is a certain amount of flexibility in the construction of the group $E$ which can be exploited to provide many more infinite hyperbolic Kazhdan groups similar to $E$. The vertex groups $E_{c_i}$ need not be cyclic: they could also be chosen to be the dihedral group  $D_{14}$ of order~$14$. One could also permute the edge groups $E_{ac_i}$ arbitrarily without changing $E_{bc_i}$. The specific choice  for the group $E$ in Theorem~\ref{thm:main} was made in order to minimize the maximal length of a relation. 

Let us note that one can also obtain larger siblings of $E$ as follows. 
For any Mersenne prime $p$, define a simple complex of groups consisting of  $p+2$ hyperbolic triangles of type $(2, 4, 6)$ glued along their hypothenuse. The two acute vertex groups are  a Heisenberg group over $\mathbf F_p$ and a dihedral group $D_{2n}$ of order $2n$, where $n = (p+1)^2 + p+2$, respectively. The other $p+2$ vertex groups are cyclic or dihedral of order $2p$. The edge groups are chosen using Proposition~\ref{prop:Dihedral} and Corollary~\ref{cor:Heisen} so that the Link Condition is satisfied at every vertex. We need $p$ to be a Mersenne prime since $p+1$ must be a prime power for Proposition~\ref{prop:Dihedral}  to apply. The fundamental group of that complex is always hyperbolic, and it has property (T) for all $p \geq 7$ by Proposition~\ref{prop:T}.

We finish this note by recording   another observation that follows from combining Proposition~\ref{prop:Dihedral} with M.~Bourdon's construction from \cite[\S1.5.3]{Bourdon} and its extension due to J.~Swiatkowski \cite{Swiat}. 

\begin{prop}\label{prop:nonDesarguesian}
Let $L$ be the incidence graph of a finite generalized $n$-gon of order $(s, t)$ with $n \geq 3$. Assume that $t$ is a prime power. 

Then there is a group $\Gamma$ acting faithfully, properly and cocompactly (but not type preservingly) on a thick locally finite triangle building $X$ of type $(3, n, n)$ admitting $L$ as the link of a vertex. 
\end{prop}
\begin{proof}
	We follow the construction described in \cite[\S5.3]{Swiat} in order to build $\Gamma$ as the fundamental group of a simple complex of finite groups. The underlying complex $\mathcal Y$ is the simplicial cone over the graph $L$. Let $V = V_1 \cup V_2$ be the bipartition of the vertex set of $L$, so that every edge in $L$ joins a vertex in $V_1$ to a vertex in $V_2$, every vertex in $V_1$ has degree $s+1$ and every vertex in $V_2$ has degree $t+1$. To each vertex $v$ in $V_2$, we define the vertex group $\Gamma_v$ as a dihedral group of order $2(t^2 +t +1)$. To each edge $e$   belonging to the set $E_L(v)$ of edges of $L$  emanating from $v$, we define $\Gamma_e$ as a cyclic group of order $2$. For all $e \in E_L(v)$ we define the inclusion of $\Gamma_e$ into $\Gamma_v$ in such a way that the coset graph of $\Gamma_v$ with respect to $\{\Gamma_e \mid e \in E_L(v)\}$ is the first barycentric subdivision of the incidence graph of the Desarguesian projective plane of order $t$. Such a choice is possible in view of Proposition~\ref{prop:Dihedral}; this is where we use the hypothesis that $t$ is a prime power. For $v \in V_1$ we define the vertex group $\Gamma_v$ to be cyclic of order~$2$, and identify $\Gamma_v$ with all edge groups $\Gamma_e$ with $e \in E_L(v)$. The groups attached to all the other simplices of $\mathcal Y$ are trivial. By \cite[Theorem~II.12.28]{BH}, the simple complex of groups defined in this way is developable. By \cite[Theorem~1]{Tits}, the univseral cover $\tilde{\mathcal Y}$ is a non-thick triangle building of type $(2, 6, n)$. Upon discarding the edges of $\tilde{\mathcal Y}$ that cover edges of $L$, we may view $\tilde{\mathcal Y}$ is a thick triangle building of type $(3, n, n)$ on which $\Gamma$ acts faithfully, properly and cocompactly, but not type-preservingly. 
\end{proof}	

The difference between  Bourdon's construction \cite[\S1.5.3]{Bourdon} and Proposition~\ref{prop:nonDesarguesian} is that the former yields Fuchsian buildings of type $(2,n,n)$. 

\begin{rmk}
 Proposition~\ref{prop:nonDesarguesian} comes close to a solution of a problem posed by W.~Kantor \cite[Problem~C.6.7]{Kantor_GABs}. It notably implies that,  all finite projective planes satisfying the   Prime Power Conjecture  appear as residue planes in $\tilde A_2$-buildings admitting a discrete cocompact group of automorphisms. In particular, all known non-Desarguesian finite projective planes do. This provides a construction of an infinite family of cocompact lattices in \emph{exotic} $\tilde A_2$-buildings of arbitrarily large thickness, where \textbf{exotic} means non-isomorphic to the Bruhat--Tits building of a simple algebraic group over a local field. In particular, the main result of \cite{BCL} applies to those lattices, which ensures that they do not admit any finite-dimensional representation with infinite image over any field. The first construction of an infinite family of cocompact lattices in  {exotic} $\tilde A_2$-buildings was obtained in \cite[Appendix~D]{BCL}; since then another source of cocompact lattices in exotic $\tilde A_2$-buildings of arbitrarily large thickness has been identified by N.~Radu \cite{Radu_PhD}. The first example of a cocompact lattice in an $\tilde A_2$-building admitting non-Desarguesian residue planes is due to him \cite{Radu}.  That example remains the only known $\tilde A_2$-building with a cocompact lattice where \emph{all} residue planes are non-Desarguesian. 
\end{rmk}

\section*{Acknowledgements}
I  thank the Isaac Newton Institute for Mathematical Sciences, Cambridge for support and hospitality during the programme \textit{Non-positive curvature group actions and cohomology} where part of the work on this paper was accomplished. I am grateful to Cornelia Drutu for asking me the question that initiated this note,  to Marc Bourdon for explaining me his inspiring construction from \cite[\S1.5.3]{Bourdon}, and to U.~Bader, M.~Kapovich and H.~Wilton for their comments on an earlier version of the manuscript.


\providecommand{\bysame}{\leavevmode\hbox to3em{\hrulefill}\thinspace}
\providecommand{\MR}{\relax\ifhmode\unskip\space\fi MR }
\providecommand{\MRhref}[2]{%
	\href{http://www.ams.org/mathscinet-getitem?mr=#1}{#2}
}
\providecommand{\href}[2]{#2}

\end{document}